\documentclass{article}
\usepackage{amsmath,amssymb,amsthm}
\newtheorem{Thm}{Theorem}
\newtheorem{Lem}{Lemma}
\newtheorem{Prop}{Proposition}
\theoremstyle{definition}

\newtheorem{Rem}{Remark}

\newtheorem{Cor}{Corollary}
\begin{document}

\title{On the maximum value of ground states for the scalar field equation with double power nonlinearity}
\author{Shinji Kawano}
\date{}
\maketitle

\begin{abstract}

We evaluate the maximum value of the unique positive solution to
\begin{equation*}
 \begin{cases}
   \triangle u+f(u)=0  & \text{in  $\mathbb{R}^n$},\\
   \displaystyle \lim_{\lvert x \rvert \to \infty} u(x)  =0, 
 \end{cases}  
\end{equation*}
where  
\begin{equation*}
f(u)=-\omega u + u^p - u^q, \qquad  \omega>0, ~~q>p>1. 
\end{equation*}
It is known that a positive solution to this problem exists if and only if $F(u):=\int_0^u f(s)ds >0$ for some $u>0$.
Moreover, Ouyang and Shi in 1998 found that the solution is unique.
In the present paper we investigate the maximum value of the solution.
The key idea is to examine the function defined from the nonlinearity, which arises from the well-known Pohozaev identity.
\end{abstract}

\section{Introduction}

We shall consider a boundary value problem
\begin{equation}
 \begin{cases}
   u_{rr}+ \dfrac{n-1}{r}u_r+f(u)=0 & \text{for $r>0$}, \\
   u_r(0)=0, \\
   \displaystyle \lim_{r \to \infty} u(r) =0, \label{b}
 \end{cases} 
\end{equation} 
where $n \in \mathbb{N}$ and
\begin{equation*}
f(u)=-\omega u + u^p - u^q, \qquad \omega>0, ~~q>p>1. 
\end{equation*}              
The above problem arises in the study of 
\begin{equation}
 \begin{cases}
   \triangle u+f(u) =0  & \text{in  $\mathbb{R}^n$},\\
   \displaystyle \lim_{\lvert x \rvert \to \infty} u(x)  =0. \label{a}
 \end{cases}  
\end{equation}
Indeed, the classical work of Gidas, Ni and Nirenberg~\cite{G1,G2} tells us 
that any positive solution to the problem~\eqref{a} is radially symmetric. 
On the other hand, for a solution $u(r)$ of the problem~\eqref{b}, $v(x):=u(\lvert x \rvert)$ is a solution to the problem~\eqref{a}.  

The condition to assure the existence of positive classical solutions to the problem~\eqref{a} (and so \eqref{b}) was given by 
Berestycki and Lions~\cite{B1} and Berestycki, Lions and Peletier~\cite{B2}:

\begin{Prop}
A positive solution to the problem~\eqref{b} exists if and only if 
\begin{equation*}
F(u):=\int_0^u f(s)ds >0, \qquad \text{for some} \quad u>0,       
\end{equation*} 
which is equivalent to
\begin{equation}
\omega < \omega_{p,q},     \label{existence}
\end{equation}
 where 
\begin{equation*}
\omega_{p,q}=\dfrac{2(q-p)}{(p+1)(q-1)} \left[ \dfrac{(p-1)(q+1)}{(p+1)(q-1)} \right] ^{\frac{p-1}{q-p}}. 
\end{equation*}
\end{Prop}
For exponent $\omega_{p,q}$, see Ouyang and Shi~\cite{OS} and Fukuizumi~\cite{Fukuizumi}. 
See also Wei and Winter~\cite{WW} and Kawano~\cite{Kawano2}.
    
The uniqueness of the positive solutions to the problem~\eqref{b} had long remained unknown.
Finally in 1998 Ouyang and Shi~\cite{OS} proved uniqueness. 

\begin{Prop}
The positive solution to the problem~\eqref{b} is unique under the condition~\eqref{existence}.
\end{Prop}

Note that for special cases when $q=2p-1$ and $\omega$ is near the exponent $\omega_{p,2p-1}=\dfrac{p}{(p+1)^2}$, 
the uniqueness result comes directly from the classical paper Peletier and Serrin~\cite{PS}. 
This fact was first remarked by Kawano~\cite{Kawano}.

\textbf{Throughout this paper we assume the condition~\eqref{existence}.}

Once the existence and the uniqueness of the solution is settled, we go on to investigate the more delicate nature of the solution. 
Here we evaluate the maximum value of the solution that is the $L^{\infty}$~norm of the solution $\| u \|_{\infty} $.

First we notice that the condition~\eqref{existence} means that 
there exist two positive exponents $0<\beta<\theta$ such that 
$F<0$ in $(0, \beta)$, $F>0$ in $(\beta, \theta)$, and then $F<0$ in $(\theta, \infty)$.
For the classification of these "double power functions" like $f$ and $F$, See Kawano~\cite{Kawano2}.

From the condition~\eqref{existence}, we know that the nonlinearity $f$ itself has positive part;
for if $f$ remains negative then $F$ remains negative, which contradicts the condition~\eqref{existence}. 
So there exist two positive exponents $0<b<c$ such that 
$f<0$ in $(0, b)$, $f>0$ in $(b, c)$, and then $f<0$ in $(c, \infty)$.

We have by simple argument that 
\begin{equation}
b<\beta<c<\theta.     \label{fF}
\end{equation}
Note that these exponents depend on $\omega$, $p$ and $q$, but independent of the dimension~$n$. 

\begin{Rem}
When $q=2p-1$, we can briefly calculate the accurete value of $b, c, \beta, \theta$.
See, for example, Kawano~\cite{Kawano}.
\begin{itemize}
\item $\displaystyle b=\left[ \frac{1-\sqrt{1-4\omega}}{2}\right]^{\frac{1}{p-1}}$.
\item $\displaystyle c=\left[ \frac{1+\sqrt{1-4\omega}}{2}\right]^{\frac{1}{p-1}}$.
\item $\displaystyle \beta=\left[ \frac{p}{p+1} \left(1-\sqrt{1-\frac{(p+1)^2}{p}\omega} \right)\right]^{\frac{1}{p-1}}$.
\item $\displaystyle \theta=\left[ \frac{p}{p+1} \left(1+\sqrt{1-\frac{(p+1)^2}{p}\omega} \right)\right]^{\frac{1}{p-1}}$.
\end{itemize}
\end{Rem}

The following basic fact is well-known. See Peletier and Serrin~\cite{PS}.
\begin{Prop}\label{exp}
Let $u(r)$ be a solution to the problem~\eqref{b}.
Then 

1)~~$u$ is strictly decreasing function of ~$r$ and decreases exponentially at infinity.
Moreover

2)~~$f(u(0))>0$ and $F(u(0))>0$.
\end{Prop}

From this result and the inequality~\eqref{fF} comes the basic estimates of $\| u \|_{\infty}$.

\begin{Prop}
\begin{equation*}
\beta<\| u \|_{\infty}<c.
\end{equation*}
\end{Prop}       

Following two theorems are the main results of the present paper:
\begin{Thm}\label{p}
Define
\begin{equation*}
\Sigma(u)=2nF(u)-(n-2)f(u).
\end{equation*}
Then there exist $(0<)~B<C~(\le \infty)$, which depends not only on $\omega$, $p$ and $q$ but also on the dimension $n$,
such that $\Sigma<0$ in $(0,B)$, and $\Sigma>0$ in $(B, C)$.
\end{Thm}     

\begin{Thm}\label{in}
\begin{equation*}
B <  \| u \|_{\infty} .
\end{equation*}
\end{Thm}   

\begin{Rem}
When $n \ge 3$, we have $B>\beta$. In this case, our theorem is the best estimate ever known. 

When $n=2$, we have $B=\beta$. 

When $n=1$, we have $B<\beta$.
\end{Rem}

To prove the inequality $B>\beta$ when $n \ge 3$, for instance, we only need to check that 
\begin{equation*}
\Sigma(\beta)=-(n-2)\beta f(\beta) <0 < -(n-2)\theta f(\theta)=\Sigma(\theta).
\end{equation*}
Then the graph of the function $\Sigma$ tells us the relation. The proof of the other cases is just the same.

\begin{Rem}
When $q=2p-1$, we can briefly calculate the accurete value of $B$ and $C$.
See Section~3.
\end{Rem}

As to such reserch, examining the maxim value of of the solution, 
We already have the following result for "single power" nonlinearity, that is the problem with 
\begin{equation*}
f_1(u)=-u+u^p , ~~1<p<p^* (n):=\begin{cases}
\, \infty & \quad(n=1,2)\\ 
\,\frac{n+2}{n-2} & \quad(n \ge 3)
\end{cases},
\end{equation*}
we have the recent result by Felmer et al.~\cite{Felmer}.
There we know that the $L^{\infty}$~norm of the unique positive solution of the problem~\eqref{b} with $f_1$ is greater than 
\begin{equation*}
\sigma_{n,p}:=\Big[ \dfrac{2(p+1)}{2n-(n-2)(p+1)} \Big]^{\frac{1}{p-1}} >0
\end{equation*}
which is the unique positive zero of the polynomial 
\begin{equation*}
\Sigma_1 (u):=2nF_1 (u)-(n-2)uf _1(u), ~~F_1(u):=\int_0^u f_1(s)ds.
\end{equation*}
It is apparent that $\Sigma_1 >0$ in $(0, \sigma_{n,p})$, and $\Sigma_1 <0$ in $(\sigma_{n,p}, \infty)$.
This polynomial is related to the well-known Pohozaev identity~\cite{P}.
We shall review in the following sections what this polynomial is for.
This paper relies on the same method, though the analysis of the polynomial is rather bothersome.
Furthermore, in that paper~\cite{Felmer}, the monotonicity of the maximum value with respect to the power $p$ is proved. 

This paper is organized as follows. 
In section~2 we give an analysis of the polynomial $\Sigma(u)$, where the proof for Theorem~\ref{p} is obtained.
In section~3 we give an explamation of the Pohozaev identity, alternative proof of the theorem~\ref{p},
also providing the proof for Theorem~\ref{in}.

In section~4 we explain the technical lemma.   

\section{Analysis of the polynomial $\Sigma(u)$: Proof of Theorem~\ref{p}. }

In this section we study the function 
\begin{equation*}
\Sigma(u)=2nF(u)-(n-2)uf(u),
\end{equation*}
where
\begin{equation*}
f(u)=-\omega u+u^p-u^q, ~~1<p<q, ~~0<\omega< \omega_{p,q}
\end{equation*}
and
\begin{equation*}
F(u)=\int_o^u f(s)ds = -\frac{\omega}{2}u^2+\frac{u^{p+1}}{p+1} -\frac{u^{q+1}}{q+1}.
\end{equation*}
So we have by calculation
\begin{equation*}
\Sigma(u)=-2\omega u^2 + 2 \sigma u^{p+1} -2 \tau u^{q+1},
\end{equation*}
where
\begin{equation*}
\sigma = \dfrac{n}{p+1}-\dfrac{n-2}{2}, ~~\tau = \dfrac{n}{q+1}- \dfrac{n-2}{2}.
\end{equation*}
By the definition we always have
\begin{equation*}
\sigma > \tau .
\end{equation*}
First we consider the case where $\tau >0$ which is equivalent to $q< p^*(n)$.
In this case we have the following equivalent condition. See Kawano~\cite{Kawano2}.
\begin{Lem}
There exist two positive exponents $B<C$ such that 

$\Sigma<0$ in $(0,B)$, ~~~$\Sigma>0$ in $(B, C)$, ~~and ~~$\Sigma<0$ in $(C, \infty)$   

$\Longleftrightarrow$
\begin{equation*}
     \omega < \sigma \dfrac{q-p}{q-1} \left[  \dfrac{\sigma (p-1)}{\tau (q-1)}  \right] ^{\frac{p-1}{q-p}} =: \omega_{\sigma,\tau , p, q}.
\end{equation*}
\end{Lem}

Tne next Lemma is the proof of Theorem~\ref{p} when $q<p^*(n)$.   
In the following section we give an alternative proof of this Lemma.
\begin{Lem}\label{label}
Let $1<p<q<p^*(n)$ and $0<\omega < \omega_{p,q}$. 

Then 
\begin{equation*}
\omega < \omega_{\sigma , \tau , p, q}.
\end{equation*}
\end{Lem}

\begin{proof}
When $n=2$, we have $\sigma = \dfrac{2}{p+1}$ and $\tau =  \dfrac{2}{q+1}$. 
In this case, $\omega_{\sigma , \tau , p, q}=\omega_{p,q}$, and the proof is over.

We show when $n \neq 2$
\begin{equation}
\omega_{p,q} < \omega_{\sigma , \tau , p, q} .     \label{pqrs}
\end{equation}

Substituting $\sigma= \frac{2n-(n-2)(p+1)}{2(p+1)} > \tau = \frac{2n-(n-2)(q+1)}{2(q+1)} >0$, 
we know that this inequality is also equivalent to
\begin{equation}
2<\frac{2n-(n-2)(p+1)}{2} \left[ \frac{2n-(n-2)(p+1)}{2n-(n-2)(q+1)}\right]^\frac{p-1}{q-p}.     \label{state}
\end{equation}
We fix $n$ and $p \in (1, p^*(n))$ and regard $q$ as a variable.
Define
\begin{equation*}
g(q)=\left[\frac{4}{2n-(n-2)(p+1)}\right]^q[2n-(n-2)(q+1)]^{p-1} 
\end{equation*}
which is a positive function on the open interval $(0, p^*(n))$.
The inequality~\eqref{state} is equivalent to
\begin{equation*}
g(q)<g(p).
\end{equation*} 
We show that the function $g$ is decreasing on $(0, p^*(n))$.
Taking logarithm of $g$ and differentiating we have
\begin{equation*}
\begin{split}
\frac{g'(q)}{g(q)} & = \log \left[\frac{4}{2n-(n-2)(p+1)}\right] + \frac{-(n-2)(p-1)}{2n-(n-2)(q+1)}    \\
                   & < \log \left[\frac{4}{2n-(n-2)(p+1)}\right] + \frac{-(n-2)(p-1)}{2n-(n-2)(p+1)}    \\
                   & = \log \left[\frac{4}{2n-(n-2)(p+1)}\right] - \frac{4}{2n-(n-2)(p+1)} +1           \\
                   & <0.                                                                                \\
\end{split}
\end{equation*}
This completes the proof.

\end{proof}

\begin{Rem}
When $q=2p-1$, we can briefly calculate the accurete value of $B$ and $C$.
Note that
\begin{equation}
\omega_{\sigma, \tau, p, 2p-1}=\frac{\sigma^2}{4\tau}.
\end{equation}
\begin{itemize}
\item $\displaystyle B=\left[ \frac{\sigma - \sqrt{{\sigma}^2 - 4\omega \tau}}{2\tau} \right]^\frac{1}{p-1}$.
\item $\displaystyle C=\left[ \frac{\sigma + \sqrt{{\sigma}^2 - 4\omega \tau}}{2\tau} \right]^\frac{1}{p-1}$.
\end{itemize}
\end{Rem}

We end the proof of Theorem~\ref{p} by considering the remaining cases.

\begin{Lem}
Let $\tau \le 0$, that is $n \ge 3$ and $q \ge p^*(n)$. 
Then there exists a positive exponent $B$ such that $\Sigma<0$ in $(0,B)$, ~~~$\Sigma>0$ in $(B, \infty)$. 
\end{Lem}

\begin{proof}
We consider four cases separably, and is easily verified;

(1) $\sigma>\tau=0$ ~~~~~~~~~~~~~~~~~~~~~~~~~~~~~~(2) $\sigma >0>\tau$

(3) $\sigma=0>\tau$ ~~~~~~~~~~~~~~~~~~~~~~~~~~~~~~(4) $0>\sigma > \tau$.
\end{proof}

\begin{Rem}
When $q=2p-1$, we can briefly calculate the accurete value of $B$.
When $\tau=0$, noting that $\sigma>0$, we have
\begin{equation*}
\begin{split}
      B & = \frac{\omega}{\sigma}    \\
        & = \lim_{\tau \downarrow0 }\left[ \frac{\sigma - \sqrt{{\sigma}^2 - 4\omega \tau}}{2\tau} \right]^\frac{1}{p-1}.           \\
\end{split}
\end{equation*}
When $\tau<0$, noting that $\sqrt{{\sigma}^2 + 4\omega (-\tau)}>\sigma$ for any $\sigma \in \mathbb{R}$, we have
\begin{equation*}
\begin{split}
      B & = \left[ \frac{-\sigma + \sqrt{{\sigma}^2 + 4\omega (-\tau)}}{2(-\tau)} \right]^\frac{1}{p-1}    \\
        & = \left[ \frac{\sigma - \sqrt{{\sigma}^2 - 4\omega \tau}}{2\tau} \right]^\frac{1}{p-1}.           \\
\end{split}
\end{equation*}
\end{Rem}

\section{Pohozaev identity: Alternative Proof of Theorem~\ref{in}.}

We review the Pohozaev identity~\cite{P}. The proof is elementary. See also Felmer et al.~\cite{Felmer} and Tang~\cite{Tang}.

\begin{Lem}
Let $u(r)$ be a solution to the problem~\eqref{b}.
Define
\begin{equation*}
P(r)=r^n[u'(r)^2 + 2F(u(r))] +(n-2)r^{n-1}u(r)u'(r).
\end{equation*}
Then 
\begin{equation}
\frac{P'(r)}{r^{n-1}}= \Sigma (u(r)).    \label{POH}
\end{equation}
\end{Lem}

Using this identity~\eqref{POH}, the proof of Lemma~\ref{label} is again easily obtained. 

\begin{proof}[Alternative Proof of Lemma~\ref{label}]
Let $\omega < \omega_{p,q}$. For the unique positive solution $u(r)$ we have 
\begin{equation}
P(0) =  P(\infty) := \lim_{r \to \infty}P(r)=0,     \label{limit}
\end{equation}
as we have exponential decay property of the solution~(Proposition~\ref{exp}).

Suppose that the conclusion of the Lemma~\ref{label} is not true. 
That means $\Sigma(u)\le 0$ in the open interval $(0, \infty)$, where the equality holds at most at one positive point $u$.
(See Kawano~\cite{Kawano2} for the classificasion.)

In this case $\Sigma(u(r)) \le 0$ where the equality holds at most at one positive point $r$.
Then $P$ is a decreasing function of $r$, by the identity~\eqref{POH}.
This is a contradiction.

\end{proof}

\begin{proof}[Proof of Theorem~\ref{in}]
Suppose $u(0) \le B$ for contradiction.
Then $u(r)<B$ for all $r>0$, as we have decreasing property of the solution~(Proposition~\ref{exp}). 
This fact and Theorem~\ref{p} shows that $\Sigma(u(r)) <0$ for all $r>0$.
Then $P$ is a decreasing function of $r$, by the identity~\eqref{POH}.
This is a contradiction to the fact that $P(0)=P(\infty)=0$.
\end{proof}

As a bonus of the above proof of the Theorem, we have the positivity of $P(r)$. 
This fact was first prooved by Ouyang and Shi~\cite{OS} in more general setting.
\begin{Cor}
\begin{equation*}
P(r)>0, \qquad \text{for all} \quad r>0.
\end{equation*}
\end{Cor}

\begin{proof}
By Theorem~\ref{p} and Theorem~\ref{in} and the decreasing property of the solution(Proposition~\ref{exp}), 
We have the following nature of $P(r)$;

There exist $r_0 >0$ such that $P'>0$ in $(0, r_0)$ and $P'<0$ in $(r_0, \infty)$.

This fact and the \eqref{limit} asserts the corollary. 

\end{proof}

\end{document}